\pdfoutput=1
\RequirePackage{ifpdf}
\ifpdf 
\documentclass[pdftex]{sigma}
\else
\documentclass{sigma}
\fi

\usepackage{mathrsfs}
\usepackage[all,poly,knot,tips]{xy}

\newdir{ >}{\xybox{*{}+/2.5mm/*@{>}}} 

\newcommand{\lra}{\longrightarrow}
\newcommand{\lla}{\longleftarrow}

\newcommand{\ldual}[1]{\mathord{{\let\nolimits\relax\sideset{^\wedge}{}{#1}}}}
\newcommand{\laction}[2]{\mathord{{\let\nolimits\relax\sideset{^{#1}}{}{#2}}}}
\newcommand{\conj}[2]{\mathord{{\let\nolimits\relax\sideset{^{#1}}{}{#2}}}}

\def\CA{{\mathscr A}}
\def\CB{{\mathscr B}}
\def\CC{{\mathscr C}}
\def\CD{{\mathscr D}}

\def\CM{{\mathscr M}}
\def\CN{{\mathscr N}}

\def\CV{{\mathscr V}}

\def\CX{{\mathscr X}}

\numberwithin{equation}{section}

\newtheorem{Theorem}{Theorem}[section]
\newtheorem{Corollary}[Theorem]{Corollary}
\newtheorem{Lemma}[Theorem]{Lemma}
\newtheorem{Proposition}[Theorem]{Proposition}
\newtheorem{Conjecture}[Theorem]{Conjecture}
{ \theoremstyle{definition}
\newtheorem{Definition}[Theorem]{Definition}
\newtheorem{Example}[Theorem]{Example}
\newtheorem{Remark}[Theorem]{Remark} }

\begin{document}


\newcommand{\arXivNumber}{1503.02783}

\renewcommand{\PaperNumber}{005}

\FirstPageHeading

\ShortArticleName{Weighted Tensor Products of Joyal Species, Graphs, and Charades}

\ArticleName{Weighted Tensor Products of Joyal Species,\\ Graphs, and Charades}

\Author{Ross {STREET}}

\AuthorNameForHeading{R.~Street}

\Address{Centre of Australian Category Theory, Macquarie University, Australia}
\Email{\href{mailto:ross.street@mq.edu.au}{ross.street@mq.edu.au}}
\URLaddress{\url{http://maths.mq.edu.au/~street/}}

\ArticleDates{Received August 18, 2015, in f\/inal form January 14, 2016; Published online January 17, 2016}

\Abstract{Motivated by the weighted Hurwitz product on sequences in an algebra,
we produce a family of monoidal structures on the category of Joyal species.
We suggest a~family of tensor products for charades.
We begin by seeing weighted derivational algebras and weighted Rota--Baxter algebras
as special monoids and special semigroups, respectively, for the same monoidal structure
on the category of graphs in a monoidal additive category.
Weighted derivations are lifted to the categorical level.}


\Classification{18D10; 05A15; 18A32; 18D05; 20H30; 16T30}

\section{Introduction}\label{Intro}

The weighted Hurwitz product arises naturally in the study of weighted derivations
and weighted Rota--Baxter operators; see \cite{GKZ2008, GKZ2014} and their references.
We review these concepts from a~ca\-tegorical perspective in Section~\ref{Review}.

In Section~\ref{gimac} we present a monoidal structure on the category
$\operatorname{Gph}\CV$ of graphs in a monoidal additive category $\CV$.
We show how these derivations and operators on a monoid $A$ in $\CV$ can be
viewed as monoid and semigroup structures on particular graphs constructed from~$A$.

The main theme of the paper is to lift concepts def\/ined on algebras (or monoids)
to concepts def\/ined on monoidal categories\footnote{Some authors call this process ``categorif\/ication''.}.
The $\lambda$-weighted Hurwitz product is traditionally def\/ined on the abelian group
$A^{\mathbb{N}}$ of sequences in an algebra~$A$.
This lifts in an obvious way, for a set $\Lambda$, to a def\/inition of $\Lambda$-weighted
product on the category $\CV^{\mathbb{N}}$ of sequences of objects in
a reasonable monoidal $\CV$.
Rather than $\CV^{\mathbb{N}}$, our interest here is in the category $[\mathfrak{S},\CV]$
of $\CV$-valued Joyal species \cite{Species} and a weighted version of the convolution product, called Cauchy product in~\cite{AgMa2010}.
This is fully discussed in Sections~\ref{Ltenspecies}--\ref{weightedbimonoidalefam} and \ref{titpa},
and motivates a def\/inition of weighted categorical derivation in Section~\ref{wcd}.
The weighted tensors def\/ine an interesting family of tensor products for linear representations
of the symmetric groups which, by choosing specif\/ic weights, include the Cauchy product
and the Heisenberg product~\cite{MoreiraPhD}.

Finally, Sections~\ref{charades} and \ref{tds} suggest and explore a weighted tensor product for
species on f\/inite vector spaces (or charades~\cite{Kapr1995, UCL}).
This generalizes, in particular, the essentially classical tensor product of representations
of the general linear groups over a f\/inite f\/ield, proved braided in~\cite{GLFq}.

Further insights into the subject of this paper appear in~\cite{GaSt}.

\section{Review of weighted derivations and Rota--Baxter operators}\label{Review}

The inspiration for the family of tensor products on species came from the $\lambda$-weighted
product of Hurwitz series as discussed in \cite{GKZ2008, GKZ2014} and their references.
They begin by def\/ining a {\em derivation of weight~$\lambda$}
on an algebra $A$ over a commutative ring $k$,
with given $\lambda \in k$, to be a $k$-module morphism $d\colon A \to A$ satisfying $d(1)=0$ and
\begin{gather*}
d(ab) = d(a)b + a d(b) + \lambda d(a)d(b) .
\end{gather*}

They note the generalized Leibnitz rule
\begin{gather*}
d^n(ab)=\sum_{k=0}^n{\sum_{j=0}^{n-k}{\binom{n}{k}\binom{n-k}{j}\lambda^kd^{n-j}(a)d^{k+j}(b)}}   .
\end{gather*}

However, I prefer to write this in the form
\begin{gather}\label{genLeib}
d^n(ab)=\sum_{n=r+s+t}{\binom{n}{r,s,t}\lambda^td^{r+t}(a)d^{s+t}(b)}
\end{gather}
to emphasise the relationship to the trinomial expansion rule for $(x+y+\lambda xy)^n$.
Here
\begin{gather*}
\binom{n}{r,s,t} = \frac{n!}{r!s!t!}  .
\end{gather*}

The {\em $\lambda$-Hurwitz product} on $A^{\mathbb{N}}$ can be def\/ined by the clearly related equation
\begin{gather}\label{LambHurw}
\big(f\cdot^{\lambda}g\big)(n)=\sum_{n=r+s+t}{\binom{n}{r,s,t}\lambda^t f(r+t) g(s+t)}   .
\end{gather}

\begin{Example}
For $\lambda = 0$, $k=\mathbb{R}$ and $A$ the algebra of smooth functions $f\colon \mathbb{R}\to \mathbb{R}$
under pointwise addition and multiplication, the dif\/ferentiation function $d\colon A\to A$ is a 0-weighted
derivation by the classical Leibnitz rule.
\end{Example}

\begin{Example}
For $\lambda$ invertible, $k=\mathbb{R}$ and $A$ the algebra of functions
$f\colon \mathbb{R}\to \mathbb{R}$ under pointwise addition and multiplication, the function
$d\colon A\to A$ def\/ined by
\begin{gather*}
d(f)(x) = \frac{f(x+\lambda)-f(x)}{\lambda}
\end{gather*}
is a $\lambda$-weighted derivation.
\end{Example}

\begin{Example}\label{consecdifference}
Def\/ine $d\colon A^{\mathbb{N}} \to A^{\mathbb{N}}$ by $d(s)(n) = s(n+1) -s(n)$.
This $d$ is a 1-weighted derivation when $A^{\mathbb{N}}$ is
equipped with the pointwise addition and multiplication.
\end{Example}

\begin{Example}\label{diffseq}
Def\/ine $d\colon A^{\mathbb{N}} \to A^{\mathbb{N}}$ by $d(f)(n) = f(n+1)$.
This $d$ is a $\lambda$-weighted derivation when $A^{\mathbb{N}}$ is
equipped with the $\lambda$-Hurwitz product for any $\lambda$.
Notice that we have an algebra morphism $d^*\colon A^{\mathbb{N}} \to (A^{\mathbb{N}})^{\mathbb{N}}$
def\/ined by $d^*(f)(m)(n)= f(m+n)$. This may motivate the next def\/inition.
\end{Example}

Def\/ine $d^*\colon A\to A^{\mathbb{N}}$ by $d^*(a)(n)=d^n(a)$.
We see that the Leibnitz rule \eqref{genLeib} amounts to

\begin{Proposition}
$d^*\colon A\to A^{\mathbb{N}}$ is an algebra morphism for all $\lambda$-weighted derivations
$d$ on~$A$, where $A^{\mathbb{N}}$ has the $\lambda$-Hurwitz product.
\end{Proposition}

In fact, $A\mapsto A^{\mathbb{N}}$ is a comonad
\begin{gather}\label{cmd}
G=\big((-)^{\mathbb{N}},\varepsilon, \delta \big)
\end{gather}
on the category $\mathrm{Alg}_k$
of $k$-algebras whose Eilenberg--Moore-coalgebras are $k$-algebras $A$ equipped with a
$\lambda$-derivation, so-called {\em $\lambda$-derivation algebras};
write $\mathrm{DA}_{\lambda}$ for the category of these.
The morphism $d^*\colon A\to A^{\mathbb{N}}$ is the coaction of the comonad.

Where there is dif\/ferentiation, there should also be integration.
A {\em Rota--Baxter operator of weight $\lambda$} on a $k$-algebra $A$ is
a $k$-linear morphism $P\colon A \to A$ satisfying
\begin{gather*}
P(a)P(b) = P (P(a)b +P(a)b+\lambda ab )   .
\end{gather*}
The pair $(A,P)$ is called a {\em $\lambda$-weighted Rota--Baxter algebra}.
Write $\mathrm{RBA}_{\lambda}$ for the category of these.

\begin{Example}
For $\lambda = 0$, $k=\mathbb{R}$ and $A$ the algebra of continuous functions $f\colon \mathbb{R}\to \mathbb{R}$
under pointwise addition and multiplication, the integration function $P\colon A\to A$,
def\/ined by $P(f)(x) = \int_0^xf(t)\mathrm{d}t$, is a 0-weighted
Rota--Baxter operator by the classical integration-by-parts rule.
\end{Example}

\begin{Example}\label{parsum}
For $\lambda = 1$ and any $k$-algebra $A$, def\/ine $P\colon A^{\mathbb{N}}\to A^{\mathbb{N}}$ to take
a sequence $u$ in $A$ to its sequence $P(u)$ of partial sums
\begin{gather*}
P(u)(n) = \sum_{i=0}^{n-1}{u(i)}   .
\end{gather*}
Then $P$ is a 1-weighted Rota--Baxter operator on $A^{\mathbb{N}}$ with pointwise addition and multiplication. See \cite{Baxt1960, Rota1969}.
For $d$ the consecutive dif\/ference operator as def\/ined in Example~\ref{consecdifference},
notice that $d\circ P=1_{A^{\mathbb{N}}}$.
\end{Example}

\begin{Example}
If $Q$ is a 1-weighted Rota--Baxter operator on $A$ then $P(a) = \lambda Q(a)$ def\/ines
a~$\lambda$-weighted Rota--Baxter operator $P$ on $A$.
\end{Example}

A {\em $\lambda$-weighted derivation RB-algebra} is a $k$-algebra $A$ equipped with
a $\lambda$-weighted deriva\-tion~$d$ and a $\lambda$-weighted Rota--Baxter operator~$P$
such that $d\circ P=1_A$. Write $\mathrm{DRB}_{\lambda}$ for the category of these.

\begin{Proposition}[see \cite{GKZ2008}] \label{lift}
Let $P$ be a RB-operator of weight $\lambda$ on~$A$.
Then $A^{\mathbb{N}}$ equipped with the $\lambda$-Hurwitz product, the derivation~$d$ of
Example~{\rm \ref{diffseq}}, and~$P$ defined by
\begin{gather*}
P(f)(n) =
\begin{cases}
P(f(0)) & \mbox{for } n=0, \\
f(n-1) & \mbox{for } n>0
\end{cases}
\end{gather*}
is a $\lambda$-weighted derivation RB-algebra.
Moreover, the following square commutes
\begin{equation*}
\xymatrix{
A^{\mathbb{N}} \ar[rr]^-{P} \ar[d]_-{\mathrm{ev}_0} && A^{\mathbb{N}} \ar[d]^-{\mathrm{ev}_0} \\
A \ar[rr]_-{P} && A}
\end{equation*}
\end{Proposition}

With a little more work following Proposition~\ref{lift}, we see that the comonad~$G$~\eqref{cmd}
lifts to $\mathrm{RBA}_{\lambda}$.
In particular, with $\mathrm{V}$ denoting the forgetful functor,
we have a comonad $\bar{G}$ and a~commutative square
\begin{gather*}
\xymatrix{
\mathrm{RBA}_{\lambda} \ar[rr]^-{\bar{G}} \ar[d]_-{\mathrm{V}} && \mathrm{RBA}_{\lambda} \ar[d]^-{\mathrm{V}}  & \\
\mathrm{Alg}_{k} \ar[rr]_-{G} && \mathrm{Alg}_{k} &  }
\end{gather*}

Write $\mathrm{RBA}_{\lambda -}$ for the category of $\lambda$-weighted
Rota--Baxter algebras where we do not insist on the algebras having a~unit.
\begin{Proposition}\label{diamond}
For each $(A,\cdot,P)\in \mathrm{RBA}_{\lambda -}$, there is
an associative binary operation $a\diamond b$ defined on $A$ by
\begin{gather*}
a\diamond b = P(a)\cdot b + a\cdot P(b) + \lambda a \cdot b  .
\end{gather*}
Moreover, $T(A,\cdot, P)=(A,\diamond, P)$ defines an endofunctor
\begin{gather*}
T\colon \ \mathrm{RBA}_{\lambda -} \lra \mathrm{RBA}_{\lambda -}  ,
\end{gather*}
which is $($well-$)$copointed by a natural transformation $\gamma \colon T\Rightarrow 1_{\mathrm{RBA}_{\lambda -}}$
whose component at $(A,\cdot,P)$ is $P\colon (A,\diamond, P) \to (A,\cdot,P)$.
\end{Proposition}

\begin{Remark}\label{Lack}
The day after my seminar talk of 4 February 2015, on the material of this section and
Section~\ref{Ltenspecies}, Stephen Lack made the following comments:
\begin{enumerate}\itemsep=0pt
\item Consider the category $[\Sigma \mathbb{N},\CV]$ whose objects
are pairs $(M,d)$ consisting of an object $M$ of a nice monoidal additive category $\CV$
and an endomorphism
$d\colon M\to M$. The forgetful functor
\begin{gather}\label{undder}
\mathrm{U}\colon \ [\Sigma \mathbb{N},\CV] \lra \CV   ,
\end{gather}
taking $(M,d)$ to $M$, has a right adjoint taking $M$ to $(M^{\mathbb{N}},d)$ where
$d(f)(n) = f(n+1)$.
There is a monoidal structure on $[\Sigma \mathbb{N},\CV]$ def\/ined by
\begin{gather*}
(M,d)\otimes^{\lambda}(N,d) = (M\otimes N, d\otimes 1+1\otimes d + \lambda d\otimes d)   .
\end{gather*}
The monoids in this monoidal category are precisely $\lambda$-derivation
algebras.
Moreover, $\mathrm{U}$~\eqref{undder}~and its right adjoint form a~monoidal adjunction
which therefore def\/ines an adjunction between the categories of monoids.
This adjunction generates the comonad~$G$~\eqref{cmd} on the category $\operatorname{Mon}\CV$
of monoids in~$\CV$.

\item There is a bialgebra
structure on the polynomial algebra $k[x]$ with comultiplication the algebra
morphism $\delta \colon k[x] \to k[x,y]\cong k[x]\otimes k[x]$ def\/ined by
\begin{gather*}
\delta(x) = x + y +\lambda xy   .
\end{gather*}
Then the convolution product on the left-hand side of the canonical isomorphism
\begin{gather*}
\mathrm{Mod}_k(k[x],A)\cong A^{\mathbb{N}}
\end{gather*}
transports to the $\lambda$-Hurwitz product on $A^{\mathbb{N}}$.

\item It feels like there should be a multicategory/promonoidal/substitude structure on $[\Sigma \mathbb{N},\CV]$ for dealing with RB-algebras.
\end{enumerate}
\end{Remark}

\section{Graphs in monoidal additive categories}\label{gimac}

Let $\CV$ be a monoidal additive category.
We act as if the monoidal structure were strict.

Let $\operatorname{Gph}\CV$ be the category of directed graphs in $\CV$.
So an object has the form of a pair of parallel morphisms $s, t \colon E \lra A$
in~$\CV$; we use $s$ and $t$ for source and target morphisms in all graphs.
A~morphism $(f,\phi) \colon (A,E) \lra (B,F)$ in $\operatorname{Gph}\CV$ consists of
morphisms $f$ and $\phi$ making the following diagram commute
\begin{gather*}
\xymatrix{
A  \ar[d]_-{f} & E \ar[l]_-{s} \ar[r]^-{t} \ar[d]^-{\phi} & A \ar[d]^-{f} \\
B & F \ar[l]^-{s} \ar[r]_-{t} & B}
\end{gather*}
Write $\mathrm{ver} \colon \operatorname{Gph}\CV \!\lra\! \CV$ for the forgetful functor
taking $(A,E)$ to $A$ and write \mbox{$\mathrm{edg} \colon \operatorname{Gph}\CV \!\lra\! \CV$}
for the forgetful functor taking~$(A,E)$ to~$E$.

We will use the notation $\langle n \rangle = \{ 1, 2, \dots, n\}$.
For $R\subseteq \langle n \rangle$, write
\begin{gather*}
\chi_R \colon \ \langle n \rangle \lra \{s,t\}
\end{gather*}
for the characteristic function of $R$ def\/ined by
\begin{gather*}
\chi_R(i) =
\begin{cases}
s & \mbox{for } i\in R, \\
t & \mbox{for } i \notin R   .
\end{cases}
\end{gather*}

Choose an endomorphism $\lambda \colon I \to I$ of the tensor unit $I$ in $\CV$.
For any $f \colon A \to B$ in $\CV$, we def\/ine
$(\lambda f \colon A \to B) = (\lambda \otimes f \colon I\otimes A \to I\otimes B)$.

Given a list $(A_1,E_1),\dots, (A_n,E_n)$ of objects of $\operatorname{Gph}\CV$,
we def\/ine an $n$-fold tensor product
\begin{gather}\label{gphtensor}
{\otimes}_{1\le i \le n}^{\lambda}{(A_i,E_i)} =  \big({\otimes}_{1\le i \le n}{A_i}, {\otimes}_{1\le i \le n}{E_i} \big)   ,
\end{gather}
where
\begin{gather*}
s = \sum_{\varnothing \neq R\subseteq \langle n \rangle}{\lambda^{(\#R-1)}\chi_R(1)\otimes \dots \otimes \chi_R(n)} \qquad \text{and} \qquad t = t \otimes \dots \otimes t   .
\end{gather*}

 For $n=2$ this gives a binary tensor product
 \begin{gather*}
(A,E){\otimes}^{\lambda} (B,F) = (A\otimes B, E\otimes F)
\end{gather*}
with
\begin{gather*}
s = \lambda s\otimes s + s\otimes t + t\otimes s \qquad \text{and} \qquad t = t \otimes t   .
\end{gather*}
The unit for this tensor is the graph $(I,I)$ with $s = 0 \colon I\to I$ and $t = 1_I\colon I \to I$.

 \begin{Proposition}\label{monstrgph}
 A monoidal structure on $\operatorname{Gph}\CV$
 is defined by \eqref{gphtensor} for any given $\lambda \in \CV(I,I)$.
 Both $\mathrm{ver}$  and $\mathrm{edg} \colon \operatorname{Gph}\CV \lra \CV$
 are strict monoidal.
 \end{Proposition}
 \begin{proof} Easy calculations of the source morphisms for
 \begin{gather*}
\big((A,E){\otimes}^{\lambda} (B,F)\big){\otimes}^{\lambda} (C,G) \qquad \text{and} \qquad  (A,E){\otimes}^{\lambda} \big((B,F){\otimes}^{\lambda} (C,G)\big)
\end{gather*}
show they agree with that of the triple tensor product.
 The target morphisms obviously agree.
 What this means is that the associativity constraints for~$\CV$ lift through
 $\mathrm{ver}$ and $\mathrm{edg}$ to $\operatorname{Gph}\CV$ and are therefore coherent.
 \end{proof}

 Let $[\Sigma \mathbb{N},\CV]$ denote the category whose objects $(A, e\colon A\to A)$ consist of
 an object~$A$ of~$\CV$ equipped with an endomorphism~$e$.
 Let
 \begin{gather}\label{jay}
J\colon \ [\Sigma \mathbb{N},\CV] \lra \operatorname{Gph}\CV
\end{gather}
be the functor def\/ined by $J(A, e) = (A,A)$ with $s = e$ and $t = 1_A$;
and $Jf=(f,f)$.
Notice also that a morphism $(f,\phi) \colon (B,F) \to J(A,e)$
 in $\operatorname{Gph}\CV$ with codomain in the subcategory
 amounts to a commutative diagram
\begin{gather*}
\xymatrix{
F \ar[r]^-{t} \ar[d]_-{s} & B \ar[r]^-{f}  & A \ar[d]^-{e} \\
B \ar[rr]_-{f} & & A }
\end{gather*}
where $\phi$ is forced to be $f\circ t \colon F \to A$.
Clearly $J$ is fully faithful and the monoidal structure of Proposition~\ref{monstrgph}
restricts to a monoidal structure on $[\Sigma \mathbb{N},\CV]$ yielding~\eqref{jay}
as a strict monoidal functor.
Indeed, this is none other than the monoidal structure of Remark~\ref{Lack}, item~1.

\begin{Definition}\label{wdermon}
A {\em $\lambda$-weighted derivational monoid} in $\CV$ is a monoid $(A,d)$
in $[\Sigma \mathbb{N},\CV]$ equipped with the monoidal structure obtained as the
restriction through \eqref{jay} of that of
Proposition~\ref{monstrgph} on $\operatorname{Gph}\CV$.
An object $(A,d) \in [\Sigma \mathbb{N},\CV]$ with an associative binary operation
will be called a~{\em $\lambda$-weighted derivational semigroup}.
\end{Definition}

More explicitly, a $\lambda$-weighted derivational monoid is a monoid $A$ in $\CV$
equipped with an endomorphism $d\colon A \to A$ satisfying the $\lambda$-weighted
equation
\begin{gather*}
d \circ \mu =  \mu \circ (\lambda d\otimes d + d\otimes 1 + 1\otimes d)  ,
\end{gather*}
and the equation $d\circ \eta = 0$ (where $\eta \colon I\to A$ is the unit of~$A$).

There is an isomorphism of categories
\begin{gather*}
\mathrm{op} \colon \ \operatorname{Gph}\CV \lra \operatorname{Gph}\CV
\end{gather*}
taking $(A,E)$ to $(A,E)^{\mathrm{op}}$ for which~$A$ and~$E$ are unchanged
but $s$ and $t$ have been interchanged.

Put
\begin{gather*}
J^{\mathrm{op}} = \big([\Sigma \mathbb{N},\CV] \stackrel{J}\lra \operatorname{Gph}\CV \stackrel{\mathrm{op}}\lra \operatorname{Gph}\CV \big)   .
\end{gather*}
Like $J$, this composite $J^{\mathrm{op}}$ is fully faithful.

However, the image of $J^{\mathrm{op}}$ is \textit{not}
closed under the monoidal structure of Proposition~\ref{monstrgph}.
All we obtain on $[\Sigma \mathbb{N},\CV]$ is a structure of multicategory
(sometimes called a ``coloured operad'').
The sets of multimorphisms are def\/ined by
\begin{gather}\label{multimorph}
\mathrm{P}_{\lambda} ((A_1,p_1),\dots,(A_n,p_n); (B,p) )
= \operatorname{Gph}\CV\big(\otimes^{\lambda}_{1\le i \le n}{J^{\mathrm{op}}(A_i,p_i)}, J^{\mathrm{op}}(B,p)\big).
\end{gather}
To be more explicit, for $R\subseteq \langle n \rangle$ and $i\in \langle n \rangle$,
put
\begin{gather*}
R(i) =
\begin{cases}
1_{A_i} & \mbox{for } i\in R, \\
p_i & \mbox{for } i \notin R  .
\end{cases}
\end{gather*}
Then, an element of the set \eqref{multimorph}, a {\em multimorphism}, is a morphism
\begin{gather*}
f \colon \ A_1\otimes \dots \otimes A_n \lra B
\end{gather*}
satisfying the equation
\begin{gather*}
f \circ (p_1\otimes \dots \otimes p_n) = p\circ f \circ \sum_{\varnothing \neq R\subseteq \langle n \rangle}{\lambda^{(\#R-1)}R(1)\otimes \dots \otimes R(n)}   .
\end{gather*}

This is a case of a general process of obtaining a multicategory structure on a category
by restriction along a functor into a monoidal category.
The notion of monoid makes sense in any multicategory.
We have the following special case.

\begin{Definition}\label{wRBmon}
A {\em $\lambda$-weighted Rota--Baxter monoid} in $\CV$ is an object $(A,p)$
of $[\Sigma \mathbb{N},\CV]$ (that is, $p\colon A \to A$ in $\CV$) equipped with the
structure of semigroup on $J^{\mathrm{op}}(A,p)$ in the monoidal category
$\operatorname{Gph}\CV$ of Proposition~\ref{monstrgph}, and a unit $\eta \colon I \to A$
for the underlying semigroup $A$ in $\CV$.
\end{Definition}

This def\/inition should make the calculation of free weighted Rota--Baxter monoids possible;
compare \cite{AM2006, Cart1972,E-FG2008, Rota1969}.

To make Def\/inition~\ref{wRBmon} a little more explicit, as expected, a $\lambda$-weighted Rota--Baxter monoid $(A,p)$ is a monoid~$A$ in~$\CV$ equipped with an endomorphism $p\colon A \to A$ satisfying 
\begin{gather*}
\mu \circ (p\otimes p) = p \circ \mu \circ (\lambda 1\otimes 1 + 1\otimes p + p\otimes 1)   .
\end{gather*}

Derivations and Rota--Baxter operators are not the only sources of semigroups and monoids for the monoidal structure of Proposition~\ref{monstrgph}. The forgetful functor
\begin{gather*}
\mathrm{ve} \colon \ \operatorname{Gph}\CV \lra \CV \times \CV   ,
\end{gather*}
taking the graph $(A,E)$ to the pair $(A,E)$, is strict monoidal
and has a right adjoint $\mathrm{R}$ def\/ined by
\begin{gather*}
\mathrm{R}(X,Y) = (X, X\oplus X\oplus Y)
\end{gather*}
with $s = \mathrm{pr}_1$ (the f\/irst projection) and $t = \mathrm{pr}_2$
(the second projection). It follows that~$\mathrm{R}$ is monoidal
and hence takes monoids to monoids.

\begin{Example}
Take $\CV = \mathrm{Mod}_k$, the category of modules over a commutative
ring $k$.
For a graph $(A,E)$ in this $\CV$, we can write $e\colon a\to b$ to mean
$a,b\in A$, $e\in E$ with $s(e)=a, t(e)=b$.
For $k$-algebras $A$ and $B$, we obtain a monoid $\mathrm{R}(A,B)$
in $\operatorname{Gph}\CV$: the graph is
$\mathrm{pr}_1, \mathrm{pr}_2\colon A\oplus A \oplus B \to A$ and the multiplication
is def\/ined by:
\begin{gather*}
\left((a_1,a_2,b) \colon a_1 \to a_2\right) \cdot \left((c_1,c_2,d)\colon c_1\to c_2\right)
=
(\lambda a_1c_1 + a_1c_2 + a_2c_1, a_2c_2,bd)\colon a_1c_1 \to a_2c_2.
\end{gather*}
\end{Example}

Of course the $\CV$-functor $J$ \eqref{jay} has both adjoints if $\CV$ is complete
and cocomplete enough. In particular, the right adjoint
\begin{gather*}
K \colon \ \operatorname{Gph}\CV \lra [\Sigma \mathbb{N},\CV]
\end{gather*}
is def\/ined by taking $K(A,E)$ to be the equalizer of the two morphisms
\begin{gather*}
s^{\mathbb{N}}, t^{\mathrm{succ}} \colon \ E^{\mathbb{N}}\to A^{\mathbb{N}}
\end{gather*}
equipped with the endomorphism $e\colon K(A,E) \to K(A,E)$ induced by $E^{\mathrm{succ}}$.
Here $\mathrm{succ} \colon \mathbb{N} \to \mathbb{N}$ is the successor function $n\mapsto n+1$.
Since $J$ \eqref{jay} is strong monoidal for the monoidal structures under discussion,
the adjunction $J\dashv K$ is monoidal.
So $K$ takes semigroups to semigroups and monoids to monoids.

In particular, if $(A,p)$ is a $\lambda$-weighted Rota--Baxter
monoid in $\CV$, then $K$ takes the graph $(A,A)$ with $s=1_A$ and $t=p$ to a
$\lambda$-weighted derivational semigroup in~$\CV$.
The underlying object is the limit of the diagram
\begin{gather*}
A \stackrel{p} \lla A \stackrel{p} \lla A \stackrel{p} \lla \cdots
\end{gather*}
in $\CV$.

\begin{Example}
Taking $\CV = \mathrm{Vect}_k$ and a $\lambda$-weighted Rota--Baxter $k$-algebra
$p\colon A\to A$, we have the non-unital $\lambda$-weighted derivational $k$-algebra
\begin{gather*}
K(J(A,p)^{\mathrm{op}}) = \big\{a\in A^{\mathbb{N}} \, | \, p(a_{n+1}) = a_n \big\}
\end{gather*}
with $d(a)_n = a_{n+1}$.
The multiplication on $K(J(A,p)^{\mathrm{op}})$ is the restriction of the $\lambda$-weighted
Hurwitz multiplication on $A^{\mathbb{N}}$ arising from the non-unital algebra~$(A,\diamond)$
of Proposition~\ref{diamond}.
Moreover, $K(J(A,p)^{\mathrm{op}})$ supports a
$\lambda$-weighted Rota--Baxter operator $p$ def\/ined by $p(a)_n = p(a_n)$.
Notice too that $d\circ p = 1$.
\end{Example}

We conclude this section by describing
the promonoidal structure in the sense of Day \cite{DayConv}
with respect to which the monoidal structure of
Proposition~\ref{monstrgph} is convolution.

Let $\mathbb{G}$ denote the category whose only objects are $0$ and $1$,
with the only non-identity morphisms $\sigma , \tau \colon 1 \to 0$.
Write $I_*\mathbb{G}$ for the free $\CV$-category on~$\CV$.
Then $\operatorname{Gph}\CV = [\mathbb{G}, \CV] = [I_*\mathbb{G}, \CV]$
where the f\/irst set of square brackets means the ordinary functor category
while the second means the $\CV$-enriched functor category.
The promonoidal structure in question is technically on~$I_*\mathbb{G}$
in the $\CV$-enriched sense. However, we can look at it as consisting of
an ordinary a functor
\begin{gather*}
\mathrm{P} \colon \ \mathbb{G}^{\mathrm{op}} \times \mathbb{G}^{\mathrm{op}} \lra \operatorname{Gph}\CV
\end{gather*}
and an object $\mathrm{J} \in \operatorname{Gph}\CV$.
Of course $\mathrm{J}$ is just the graph $0, 1 \colon I \to I$ which is the tensor unit.
We can regard $\mathrm{P}$ as a ``cograph of cographs of graphs'' (although a cograph
looks just like a graph):
\begin{gather*}
\xymatrix{
I \ar@<-1ex>[rr]_-{(0,1)} \ar@<1ex>[rr]^-{(1,0)} \ar@<-1ex>[dd]_{(1,0)} \ar@<1ex>[dd]^{(0,1)} && I\oplus I=2\cdot I \ar@<-1ex>[dd]_{\bigl(\begin{smallmatrix}
1&0&0&0\\
0&0&1&0
\end{smallmatrix} \bigr)} \ar@<1ex>[dd]^{\bigl(\begin{smallmatrix}
0&1&0&0\\
0&0&0&1
\end{smallmatrix} \bigr)} \\
\\
2\cdot I=I\oplus I \ar@<-1ex>[rr]_-{\bigl(\begin{smallmatrix}
0&0&1&0\\
0&0&0&1
\end{smallmatrix} \bigr)} \ar@<1ex>[rr]^-{\bigl(\begin{smallmatrix}
1&0&0&0\\
0&1&0&0
\end{smallmatrix} \bigr)} && \left((\lambda,1,1,0), (0,0,0,1) \colon I \to 4\cdot I\right) }
\end{gather*}

\section[The $L$-Hurwitz product of species]{The $\boldsymbol{L}$-Hurwitz product of species}\label{Ltenspecies}

Let $\mathfrak{S}$ denote the groupoid whose objects are f\/inite sets
and whose morphisms are bijective functions.
We write $U+V$ for the disjoint union of sets~$U$ and~$V$;
this is the binary coproduct as objects of the category~$\mathrm{Set}$
of sets and all functions.
It is not the coproduct in~$\mathfrak{S}$; yet it does provide the
symmetric monoidal structure on $\mathfrak{S}$ of interest here.
When we write $X=A+B$ for~$A$ and~$B$ subsets of a set~$X$, we mean $X=A\cup B$ and $\varnothing = A\cap B$.

We have the particular f\/inite sets $\langle n \rangle = \{ 1, 2, \dots, n\}$.

Let $\CV$ denote a monoidal category with f\/inite coproducts which are preserved by tensoring on either side by an object. The tensor product of $V,W\in \CV$ is denoted by $V\otimes W$ and the unit object by $I$. Justif\/ied by coherence theorems (see~\cite{BTC} for example), we write as if the monoidal structure on $\CV$ were strictly associative and strictly unital.
For any set~$S$, write $S\cdot V$ for the coproduct of~$S$ copies of $V\in \CV$, when it exists
(as it does for~$S$ f\/inite).

The {\em category of $\CV$-valued Joyal species},  after \cite{Species, AnalFunct}, is the functor category
$[\mathfrak{S},\CV]$. The objects will simply be called {\em species} when
$\CV$ is understood.

Suppose $L\colon \mathfrak{S} \to \mathcal{Z}\CV$ is a braided strong monoidal functor into the
monoidal centre (in the sense of \cite{TYBO}) of $\CV$.
We have natural isomorphisms
\begin{gather}\label{u}
u_{X,V} \colon \  LX\otimes V\cong V\otimes LX,
\end{gather}
such that
\begin{gather*}
\xymatrix{
LX\otimes V\otimes W \ar[rd]_{u_{X,V\otimes W}} \ar[rr]^{u_{X,V}\otimes 1_W}   && V\otimes LX\otimes W \ar[ld]^{ \ 1_V\otimes u_{X,W}} \\
& V\otimes W\otimes LX  &
}
\end{gather*}
If $\CV$ itself is braided ({\em a fortiori} symmetric), we can take a braided strong monoidal functor
$\mathfrak{S} \to \CV$
and compose it with the canonical braided strong monoidal functor $\CV \to \mathcal{Z}\CV$
to obtain such an $L$.

By way of example of an $L\colon \mathfrak{S} \to \mathcal{Z}\CV$, we could take any f\/inite set
$\Lambda$ and $LX = \Lambda^X \cdot I$ with $L\sigma = \Lambda^{\sigma^{-1}}\cdot I$
for any bijective function $\sigma$.

\begin{Definition} The {\em $L$-Hurwitz product} $F\otimes^LG$ of species $F$ and $G$
is def\/ined on objects $X\in \mathfrak{S}$ by
\begin{gather}\label{Lten}
\big(F\otimes^LG\big)X = \sum_{X=U\cup V}{L(U\cap V) \otimes FU\otimes GV}  .
\end{gather}
The def\/inition of $F\otimes^LG$ on morphisms is clear since any
bijective function $\sigma \colon X\to Y$ restricts to bijections
\begin{gather*}
U\to \sigma U, \qquad V\to \sigma V, \qquad U\cup V \to \sigma U\cup \sigma V, \qquad U\cap V \to \sigma U\cap \sigma V .
\end{gather*}
\end{Definition}

Let $J\colon \mathfrak{S} \to \CV$ be the species whose value at $X$ is the unit $I$ for tensor in~$\CV$
when $X$ is empty and is initial in~$\CV$ otherwise.
Clearly $J$ is a unit for the $L$-Hurwitz product in the sense that we have canonical isomorphisms
\begin{gather*}
\lambda_G\colon \ J\otimes^LG \to G \qquad \text{and} \qquad \rho_F \colon \ F \to F\otimes^LJ  .
\end{gather*}
Associativity isomorphisms
\begin{gather}\label{Lassoc}
\alpha_{F,G,H} \colon \ \big(F\otimes^LG\big)\otimes^LH \cong F\otimes^L\big(G\otimes^LH\big)
\end{gather}
are obtained using the following result easily proved by Venn diagrams.

\begin{Lemma}\label{Venn}
$(U\cup V)\cap W + U\cap V \cong U\cap (V\cup W) + V\cap W$.
\end{Lemma}

Then, to def\/ine \eqref{Lassoc}, we use the isomorphisms
\begin{gather*}
 L((U\cup V)\cap W)\otimes L(U\cap V)\otimes FU\otimes GV\otimes HW  \\
\qquad{} \cong   L((U\cup V)\cap W+U\cap V)\otimes FU\otimes GV\otimes HW \\
\qquad{} \cong   L(U\cap (V\cup W) + V\cap W)\otimes FU\otimes GV\otimes HW \\
\qquad{} \cong   L(U\cap (V\cup W)) \otimes L(V\cap W)\otimes FU\otimes GV\otimes HW \\
\qquad{} \cong   L(U\cap (V\cup W)) \otimes FU\otimes  L(V\cap W)\otimes GV\otimes HW   ,
\end{gather*}
the f\/irst and third coming from the strong monoidal structure on $L$, the second from
Lemma~\ref{Venn}, and the fourth from the monoidal centre structure~\eqref{u} on $L(V\cap W)$.

In the case $L=J$, we recover from~\eqref{Lten} the usual convolution (Cauchy) product of species
appearing in~\cite{Species}.
In the case where $L$ is the exponential series $LX=I$ for all~$X$, we recover the Heisenberg
product appearing in \cite{AFM2015, MoreiraPhD}.

For a general $L$, the term $L(U\cap V)$ can be
considered a measure of the failure of $U$ and $V$ to be disjoint.

\section{A combinatorial interpretation}\label{aci}

We consider the case where $\CV = \mathrm{Set}$ so that $[\mathfrak{S},\mathrm{Set}]$
is the category of species as studied in~\cite{Species}.
Fix any set $\Lambda$. Def\/ine the species~$L$ by
\begin{gather*}
LX = \biggl\{ S=(S_{\lambda})_{\lambda \in \Lambda} \, | \, S_{\lambda}\subseteq X , \, \sum_{\lambda \in \Lambda}{S_{\lambda}} = X \biggr\}
\end{gather*}
and $(L\sigma )S=(\sigma S_{\lambda})_{\lambda \in \Lambda}$.
In other words, a structure of the species $L$ on the set $X$ is a partition of $X$ into a
$\Lambda$-indexed family of disjoint (possibly empty) subsets.

A structure of the species $F\otimes^LG$ on the set consists of a quintuplet $(U,V,S,\phi,\gamma)$
where~$U$,~$V$ are subsets of $X$ such that $X=U\cup V$, and~$S$, $\phi$, $\gamma$ are
$L$-, $F$-, $G$-structures on $U\cap V$, $U$, $V$, respectively.

We write $\#S$ for the cardinality of the set $S$. We assume $\Lambda$
is f\/inite and put $\lambda = \# \Lambda$.

The {\em cardinality sequence} of a species $F$ is the sequence $\#F \colon \mathbb{N} \to \mathbb{Z}$
def\/ined by
\begin{gather*}
(\#F)(n) = \#F\langle n \rangle   .
\end{gather*}

We consider the $\lambda$-Hurwitz product \eqref{LambHurw} on $\mathbb{Z}^{\mathbb{N}}$.

\begin{Proposition}\label{card}
$\# (F\otimes^LG)= \# F\cdot^{\lambda} \#G$.
\end{Proposition}

This result specializes to Theorem~2.7 of \cite{AFM2015} when~$L$ is the exponential species
and $\lambda = 1$.

\section{The iterated tensor and coherence}\label{titac}

\begin{Proposition}\label{altLtensor}
An alternative definition of $F\otimes^LG$ is
\begin{gather*}
\big(F\otimes^LG\big)X = \sum_{X=A+B+C}{L(C) \otimes F(A+C)\otimes G(B+C)}   .
\end{gather*}
\end{Proposition}
\begin{proof} Given $X=A+B+C$, put $U=A+C$ and $V=B+C$.
Given $X=U\cup V$, put $A= U\backslash V$, $B= V\backslash U$, and $C=U\cap V$.
\end{proof}

The $n$-fold version of this tensor product is
\begin{gather}
 \otimes^L_n(F_1,\dots,F_n)X    \nonumber\\
 \qquad{}
  =   \sum_{X=\sum\limits_{\varnothing \ne S\subseteq \langle n \rangle} {A_S}} {L\left( \sum_{S}(\#S-1)\cdot A_S\right) \otimes {F_1\left( \sum_{1\in S}A_S\right) \otimes \dots \otimes F_n\left( \sum_{n\in S}A_S\right) }}  .\label{nfoldLtensor1}
\end{gather}
This yields the formula in Proposition~\ref{altLtensor} for $n=2$ by taking $A=A_{1}$, $B=A_{2}$, $C=A_{\{1,2\}}$. Note that~\eqref{nfoldLtensor1} is unchanged if we replace $\langle n \rangle$ by
any set of cardinality~$n$.

\begin{Remark}
As Joachim Kock reminded me, if we replace $\langle n \rangle$ by the
`$(n-1)$-simplex' $[n-1] = \{0,1,\dots , n-1\}$, then the non-empty subsets~$S$
correspond to the non-degenerate faces of~$[n-1]$ and~$\#S-1$ is the dimension
of the face.
\end{Remark}

Let us consider the ef\/fect of inserting one pair of parentheses in a multiple tensor~\eqref{nfoldLtensor1}.
We look at
\begin{gather*}
\otimes^L_{p+1+r}\big(F_1,\dots,F_p, \otimes^L_q(F_{p+1},\dots,F_{p+q}),F_{p+q+1},\dots,F_{p+q+r}\big)X   .
\end{gather*}
Using \eqref{nfoldLtensor1} twice, once with $n = p+1+r$ and once with $n=q$, we obtain the expression
\begin{gather*}
 L\left(\sum_{T}(\#T-1)\cdot B_T\right)\otimes F_1\left(\sum_{1\in T}B_T\right) \otimes \dots \otimes F_p\left(\sum_{p\in T}B_T\right)  \\
\qquad{} \otimes   L\left(\sum_{R}(\#R-1)\cdot C_R\right)\otimes F_{p+1}\left(\sum_{p+1\in R}C_R\right) \otimes \dots \otimes F_{p+q}\left(\sum_{p+q\in R}C_R\right) \nonumber \\
 \qquad{} \otimes   F_{p+q+1}\left(\sum_{p+q+1\in T}B_T\right) \otimes \dots \otimes F_{p+q+r}\left(\sum_{p+q+r\in T}B_T\right) \nonumber
\end{gather*}
summed over all families
\begin{gather*}
B = \big( B_T \, | \, \varnothing \ne T\subseteq \{1,\dots , p,\star ,p+q+1,\dots , p+q+r\} \big)
\end{gather*}
providing a partition
$X= \sum_TB_T$ of $X$, together with all families
\begin{gather*}
C = \big( C_R \, | \, \varnothing \ne R\subseteq \{p+1,\dots , p+q\} \big)
\end{gather*}
providing a partition $\sum_{\star \in T}B_T= \sum_RC_R$ of $\sum_{\star \in T}$.
Using the fact that $L$ lands in $\mathcal{Z}\CV$ and that~$L$ is strong monoidal, we obtain
the isomorphic expression
\begin{gather}
 L \left(\sum_{T}(\#T-1)\cdot B_T + \sum_{R}(\#R-1)\cdot C_R \right)
  \otimes   F_1\left(\sum_{1\in T}B_T\right) \otimes \dots \otimes F_p\left(\sum_{p\in T}B_T\right) \nonumber \\
\qquad{} \otimes   F_{p+1}\left(\sum_{p+1\in R}C_R\right) \otimes \dots \otimes F_{p+q}\left(\sum_{p+q\in R}C_R\right) \nonumber \\
\qquad{}  \otimes   F_{p+q+1}\left(\sum_{p+q+1\in T}B_T\right) \otimes \dots \otimes F_{p+q+r}\left(\sum_{p+q+r\in T}B_T\right)\label{nfoldLtensor3} \end{gather}
summed over the same families $(B,C)$.
For $\star \in T$, we have $B_T = \sum_R{C_R\cap B_T}$.
On the other hand,  $C_R = \sum_{\star \in T}{C_R\cap B_T}$.
Put
\begin{gather*}
Q = \{p+1,\dots , p+q \} \qquad  \text{and}  \qquad N = \{1,\dots ,p\}\cup \{ p+q+1, \dots , p+q+r  \}  ,
\end{gather*}
and obtain a family
\begin{gather*}
A =  (A_S \, | \, \varnothing \ne S\subseteq \langle p+q+r \rangle  )
\end{gather*}
partitioning $X$ by def\/ining
\begin{gather*}
A_S =
\begin{cases}
B_S & \mbox{for} \ S\cap Q = \varnothing, \\
C_{S\cap Q} \cap B_{(S\cap N) \cup \{\star \}} & \mbox{for} \ S\cap Q \ne \varnothing   .
\end{cases}
\end{gather*}
Then we can recover the $B$ and $C$ families via
\begin{gather*}
B_T =
\begin{cases}
A_T & \mbox{for} \  \star \notin T , \\
\sum_R{A_{R\cup (T\backslash \star)}} & \mbox{for} \ \star \in T,
\end{cases}
\qquad
\text{and}
\qquad
C_R = \sum_{\star \in T}{A_{R\cup (T\backslash \star)}}  .
\end{gather*}
We have the following equations
\begin{alignat*}{3}
& (i)\quad && \sum_{S}{(\#S-1)\cdot A_S} = \sum_{T}{(\#T-1)\cdot B_T} + \sum_{R}{(\#R-1)\cdot C_R},&\\
& (ii) \quad &&
\sum_{k\in S}{A_S} =
\begin{cases}
\displaystyle \sum\limits_{k\in T}{B_T} & \mbox{for} \ 1\le k \le p \ \text{ or } \ p+q +1\le k \le p+q + r,  \vspace{1mm}\\
\displaystyle \sum\limits_{k\in R}{C_R} & \mbox{for} \ p+1\le k \le p+q  .
\end{cases}
\end{alignat*}
This shows that the sum of the expressions \eqref{nfoldLtensor3} over the pairs~$(B,C)$
is equal to \eqref{nfoldLtensor1} with $n = p+q+r$.
Remember however that the tensor product $+$ on $\mathfrak{S}$ is not strict symmetric;
the symmetry on $\mathfrak{S}$ provides canonical bijections between the left- and right-hand
sides of~(i) and~(ii).
Since~$L$ is braided, we have constructed a natural isomorphism
\begin{gather}
 a_{p,q,r} \colon  \ \otimes^L_n(F_1,\dots,F_{p+q+r})  \nonumber \\
  \hphantom{a_{p,q,r} \colon  \ }{} \cong
\otimes^L_{p+1+r}\big(F_1,\dots,F_p, \otimes^L_q(F_{p+1},\dots,F_{p+q}),F_{p+q+1},\dots,F_{p+q+r}\big)   . \label{nfoldassoc}
\end{gather}

Now consider the Mac Lane--Stashef\/f pentagon for 2-fold bracketings of
$F_1\otimes^L F_2\otimes^L F_3 \otimes^L F_4$ as the vertices.
Let $a\colon H \to K$ denote one of the edges of the pentagon obtained using
the associativity isomorphisms~\eqref{Lassoc}.
There is a composite~$b$ of two isomorphisms, each using one instance
of an isomorphism \eqref{nfoldassoc}, which goes from
$\otimes^L_4(F_1,F_2,F_3,F_4)$ to~$H$, and another one
$c \colon \otimes^L_4(F_1,F_2,F_3,F_4) \to H$.
By coherence of the braided strong monoidal functor~$L$, it follows that $a\circ b = c$.
Commutativity of the pentagon is a consequence of commutativity of all
these triangular sides of the so-formed pentagonal cone.

\section[Promonoidal structures on $\mathfrak{S}$]{Promonoidal structures on $\boldsymbol{\mathfrak{S}}$}\label{psoS}

For f\/inite sets $A$, $B$ and $X$, let $\mathrm{Cov}(A,B;X)$ denote the set of jointly surjective pairs
$(\mu , \nu)$ of injective functions
\begin{gather*}
A\stackrel{\mu} \lra X \stackrel{\nu} \lla B   .
\end{gather*}
We write $A\times_X B$ for the pullback of $\mu$ and $\nu$.

Def\/ine a functor
\begin{gather*}
\mathrm{P} \colon \ \mathfrak{S}^{\mathrm{op}}\times \mathfrak{S}^{\mathrm{op}}\times \mathfrak{S} \lra \CV
\end{gather*}
by
\begin{gather*}
\mathrm{P}(A,B;X) = \sum_{(\mu , \nu)\in \mathrm{Cov}(A,B;X)}{L(A\times_X B)}   .
\end{gather*}

\begin{Proposition}
$(F\otimes^L G)X \cong \int^{A,B}{\mathrm{P}(A,B;X)\otimes FA\otimes GB}$.
\end{Proposition}

\begin{proof}
A universal dinatural transformation
\begin{gather*}
\theta_{A,B} \colon \ \mathrm{P}(A,B;X) \otimes FA\otimes GB \lra \sum_{X=U\cup V}{L(U\cap V) \otimes FU\otimes GV}
\end{gather*}
is def\/ined by taking its composite with the injection at $(\mu , \nu)\in \mathrm{Cov}(A,B;X)$
to be obtained from the $(\mu(A),\nu(B))$ injection and the bijections $A\cong \mu(A)$,
$B\cong \nu(B)$, $A\times_X B \cong \mu(A)\cap \nu(B)$, noting $X = \mu(A)\cup \nu(B)$.
\end{proof}

By Day's general theory of promonoidal categories \cite{DayPhD, DayConv}, we have
\begin{Corollary}
If moreover $\CV$ is $($left and right$)$ closed and sufficiently complete then $\otimes^L$ defines a~$($left and right$)$ closed monoidal structure on $[\mathfrak{S},\CV]$.
The monoidal structure coincides with that of Section~{\rm \ref{Ltenspecies}}.
\end{Corollary}

\section[The weighted bimonoidale structure on $\operatorname{fam}\mathfrak{S}$]{The weighted bimonoidale structure on $\boldsymbol{\operatorname{fam}\mathfrak{S}}$}\label{weightedbimonoidalefam}

Lately (for example, in~\cite{103}), we have used the term {\em monoidale} for
``pseudomonoid'', also called ``monoidal object'', in a monoidal bicategory~$\CM$~\cite{mbaHa}.
For example, the monoidales in the cartesian monoidal bicategory~$\mathrm{Cat}$
are monoidal categories.

When the monoidal bicategory $\CM$ is symmetric, the monoidales themselves form a
symmetric monoidal bicategory where the morphisms are strong monoidal.
With the same tensor product, the opposite bicategory~$\CM^{\mathrm{op}}$ is symmetric
monoidal.
A~{\em bimonoidale in $\CM$} is a monoidale in~$\CM^{\mathrm{op}}$.
Incidentally, every monoidale in $\mathrm{Cat}$ is uniquely a bimonoidale.

Consider the 2-category $\mathrm{Cat}_{+}$ of (small) categories admitting f\/inite coproducts,
and f\/inite-coproduct-preserving functors. This becomes a symmetric closed monoidal
bicategory (see~\cite{mbaHa}) with tensor product $\CA \boxtimes \CB$ representing functors
$H \colon \CA \times \CB \to \CX$ for which each $H(A,-)$ and each $H(-,B)$ is f\/inite
coproduct preserving.
Clearly the monoidal category~$\CV$ of Section~\ref{Ltenspecies} is a~monoidale
(= pseudomonoid) in $\mathrm{Cat}_{+}$.

For any category $\CC$, we write $\operatorname{fam}\CC$ for the free f\/inite coproduct
completion of $\CC$. That is, $\mathrm{fam}$ provides the left biadjoint to the forgetful
2-functor $\mathrm{Cat}_{+} \to \mathrm{Cat}$. Indeed, $\mathrm{fam}$ is a strong monoidal
pseudofunctor; in particular, there is a canonical equivalence
\begin{gather*}
\operatorname{fam}\CC \boxtimes \operatorname{fam}\CD \simeq \operatorname{fam} (\CC \times \CD  )   .
\end{gather*}
Every monoidal category $\CC$ determines a monoidale $\operatorname{fam}\CC$ in $\mathrm{Cat}_{+}$.

Explicitly, the objects of $\operatorname{fam}\CC$ can be written formally as $\sum\limits_{s\in S}{C_s}$
where $S$ is a f\/inite set and $C_s\in \CC$.
Then, if $\CC$ is monoidal, the monoidale structure on $\operatorname{fam}\CC$ is def\/ined by
\begin{gather*}
\sum_{s\in S}{C_s}\otimes \sum_{t\in T}{D_t} = \sum_{(s,t)\in S\times T}{C_s\otimes D_t}   .
\end{gather*}

We are interested in $\operatorname{fam}\mathfrak{S}$.
By what we have just said, this is a monoidale in $\mathrm{Cat}_{+}$:
\begin{gather*}
\sum_{s\in S}{U_s}\otimes \sum_{t\in T}{V_t} = \sum_{(s,t)\in S\times T}{(U_s+V_t)}   .
\end{gather*}

Fix a f\/inite set $\Lambda$ and def\/ine $L \colon \mathfrak{S} \to \mathrm{Set}$ by $LX = \Lambda^X$
and $L\sigma = \Lambda^{\sigma^{-1}}$. Def\/ine a coproduct-preserving functor
\begin{gather}\label{Delta}
\Delta \colon \ \operatorname{fam}\mathfrak{S} \lra \operatorname{fam}(\mathfrak{S}\times \mathfrak{S}) \simeq \operatorname{fam}\mathfrak{S}\boxtimes \operatorname{fam}\mathfrak{S}
\end{gather}
 by
 \begin{gather*}
\Delta(X) = \sum_{X=A+B+C}{L(C)\cdot (A+C,B+C)}
\end{gather*}
for $X\in \mathfrak{S}$.

\begin{Proposition}
The functor $\Delta$ of \eqref{Delta} is strong monoidal.
\end{Proposition}
\begin{proof}
In $\Delta(X+Y)= \sum_{X+Y=A+B+C}{L(C)\cdot (A+C,B+C)}$ we can put
\begin{gather*}
P=X\cap A,\qquad Q=X\cap B, \qquad R=X\cap C,\\  U=Y\cap A, \qquad V=Y\cap B, \qquad W=Y\cap C
\end{gather*}
to obtain
\begin{gather*}
 \Delta(X+Y)
  =    \sum_{X =P+Q+R, Y=U+V+W}{L(R+W)\cdot (P+U+R+W,Q+V+R+W)} \\
 \hphantom{\Delta(X+Y) }{}
  \cong   \sum_{X =P+Q+R}{L(R)\cdot (P+R,Q+R)\times \sum_{Y =U+V+W}{L(W)\cdot (U+W,V+W)}} \\
 \hphantom{\Delta(X+Y) }{}
  \cong   \Delta X \times \Delta Y   ,
\end{gather*}
as required.
\end{proof}

The relationship between this structure and the promonoidal structure of Section~\ref{psoS}
will be examined elsewhere; indeed, see~\cite{GaSt}.

\section{Weighted categorical derivations}\label{wcd}

Suppose $\CV$ is a symmetric monoidal closed category which is complete and cocomplete,
and suppose $L\colon \mathfrak{S} \to \CV$ is a strong monoidal functor.

Harking back to Remark~\ref{Lack}, we are prompted to consider the 2-category
\begin{gather}\label{Hom}
\mathfrak{E} = \mathrm{Hom}(\Sigma \mathfrak{S},\CV\text{-}\mathrm{Cat}_{L,+})   .
\end{gather}
Here $\Sigma \mathfrak{S}$ denotes the bicategory with one object (denoted $\star$)
whose homcategory is the symmetric groupoid $\mathfrak{S}$; composition
is provided by the monoidal structure~$+$ on~$\mathfrak{S}$.
Also~$\CV\text{-}\mathrm{Cat}_{L,+}$ denotes the 2-category of $\CV$-categories
admitting f\/inite coproducts and tensoring with the object~$L(X)$ of~$\CV$;
the morphisms are $\CV$-functors preserving these colimits;
the 2-cells are $\CV$-natural transformations.
The objects of~\eqref{Hom} are pseudofunctors $T \colon \Sigma \mathfrak{S}\to \CV\text{-}\mathrm{Cat}_{L,+}$,
the morphisms are pseudonatural transformations, and the 2-cells are modif\/ications
(in terminology of~\cite{KelSt1974}).
Such an object $T$ determines a $\CV$-category $T\star = \CM \in \CV\text{-}\mathrm{Cat}_{L,+}$
and a strong monoidal functor
$T_{\star \star}\colon \mathfrak{S} \to \CV\text{-}\mathrm{Cat}_{L,+}(\CM,\CM)$.
This $T_{\star \star}$ is determined up to equivalence by an endomorphism
$D\colon \CM \to \CM$ in $\CV\text{-}\mathrm{Cat}_{L,+}$
and an involutive Yang--Baxter\footnote{This is Rodney Baxter
\url{http://en.wikipedia.org/wiki/Rodney_Baxter}, not the author of~\cite{Baxt1960}.} operator
$\rho \colon D\circ D \Rightarrow D\circ D$ on $D$
(for example, see~\cite{TYBO} for terminology).
Then $T_{\star \star}\langle n \rangle \cong D^{\circ n}$ and, for the
non-identity bijection $\tau \colon \langle 2 \rangle \to \langle 2 \rangle$,  $T\tau$ transports to $\rho$.
Therefore we shall write the object~$T$ of~$\mathfrak{E}$~\eqref{Hom} as a pair $(\CM, D^*)$
where $T\star = \CM$ and $T_{\star \star} = D^*$.
The morphisms of $\mathfrak{E}$ are then squares
\begin{gather*}
\xymatrix{
\CM \ar[d]_{D^*X}^(0.5){\phantom{AA}}="1" \ar[rr]^{K}  && \CN \ar[d]^{E^*X}_(0.5){\phantom{AA}}="2" \ar@{=>}"1";"2"^-{\kappa_X \cong}
\\
\CM \ar[rr]_-{K} && \CN
}
\end{gather*}
in $\CV\text{-}\mathrm{Cat}_{L,+}$ which are $\CV$-natural in $X$ and, stacking vertically, respect the tensor in $\mathfrak{S}$.
Ge\-ne\-ralizing the tensor $\boxtimes$ on $\mathrm{Cat}_{+}$ as in Section~\ref{weightedbimonoidalefam},
we have a tensor, also denoted by $\boxtimes$, on $\CV\text{-}\mathrm{Cat}_{L,+}$, where the
tensor product $\CA \boxtimes \CB$ represents $\CV$-functors
$H \colon \CA \otimes \CB \to \CX$ for which each of~$H(A,-)$ and~$H(-,B)$ preserves f\/inite
coproducts and tensoring with each~$L(X)$.
This makes $\CV\text{-}\mathrm{Cat}_{L,+}$ into a monoidal bicategory.

This tensor product $\boxtimes$ on $\CV\text{-}\mathrm{Cat}_{L,+}$ lifts to one,
denoted by $\widehat{\boxtimes}$, on $\mathfrak{E}$ \eqref{Hom}:
\begin{gather*}
(\CM, D^*)\widehat{\boxtimes} (\CN, E^*) = \big(\CM \boxtimes \CN, D^*\widehat{\boxtimes}E^*\big)m,
\end{gather*}
where
\begin{gather*}
\big(D^*\widehat{\boxtimes}E^*\big)X = \sum_{X=U\cup V}{L(U\cap V) \otimes
D^*U \boxtimes E^*V} .
\end{gather*}
To see that $D^*\widehat{\boxtimes}E^* \colon \mathfrak{S}
\to \CV\text{-}\mathrm{Cat}_{L,+}(\CM \boxtimes \CN,\CM \boxtimes \CN)$
is strong monoidal, we calculate
\begin{gather*}
 (D^*\widehat{\boxtimes}E^*)(X+Y)  \cong  \sum_{X+Y=U\cup V}{L(U\cap V) \otimes
D^*U \boxtimes E^*V} \\
\qquad{} \cong   \sum_{X=U_1\cup V_1, Y = U_2\cup V_2}{L(U_1\cap V_1 + U_2\cap V_2) \otimes
(D^*U_1\circ D^*U_2) \boxtimes (E^*V_1\circ E^*V_2)} \\
\qquad{} \cong \sum_{X=U_1\cup V_1, Y = U_2\cup V_2}{L(U_1\cap V_1) \otimes L(U_2\cap V_2) \otimes
(D^*U_1\boxtimes E^*V_1) \circ (D^*U_2\boxtimes E^*V_2)}\\
\qquad{} \cong \sum_{X=U_1\cup V_1, Y = U_2\cup V_2}{L(U_1\cap V_1) \otimes
(D^*U_1\boxtimes E^*V_1) \circ L(U_2\cap V_2)\otimes (D^*U_2\boxtimes E^*V_2)} \\
\qquad{} \cong (D^*\widehat{\boxtimes}E^*)X \circ (D^*\widehat{\boxtimes}E^*)Y   .
\end{gather*}
In this way, $\mathfrak{E}$ \eqref{Hom} becomes a monoidal bicategory.

\begin{Definition} An {\em $L$-weighted derivation} $D^*$ on a monoidale $\CM$ in $\CV\text{-}\mathrm{Cat}_{L,+}$
is a lifting of the monoidale structure on $\CM$ to a monoidale structure on
$(\CM,D^*)$ in $\mathfrak{E}$~\eqref{Hom}.
\end{Definition}

\begin{Example}\label{specder}
An $L$-weighted derivation $D^* \colon \mathfrak{S} \to \CV\text{-}\mathrm{Cat}_{L,+}\left( [\mathfrak{S}, \CV],[\mathfrak{S}, \CV] \right)$ on the monoidale $([\mathfrak{S}, \CV], \otimes^L)$
is def\/ined by $(D^*X)F = F(X+-)$. The main point is the canonical isomorphism below
\begin{gather*}
\xymatrix{
[\mathfrak{S}, \CV]\boxtimes [\mathfrak{S}, \CV] \ar[d]_{(D^*\widehat{\boxtimes}D^*)X}^(0.5){\phantom{AAAAAA}}="1" \ar[rr]^{\otimes^L}  && [\mathfrak{S}, \CV] \ar[d]^{D^*X}_(0.5){\phantom{AAAAAA}}="2" \ar@{=>}"1";"2"^-{\cong}
\\
[\mathfrak{S}, \CV]\boxtimes [\mathfrak{S}, \CV] \ar[rr]_-{\otimes^L} && [\mathfrak{S}, \CV]
}
\end{gather*}
 \end{Example}

\begin{Remark}
The f\/irst item of Remark~\ref{Lack} has a categorical version.
The forgetful 2-functor $\mathrm{U} \colon \mathfrak{E} \to \CV\text{-}\mathrm{Cat}_{L,+}$ has
a right biadjoint $\mathrm{JS}$ taking the $\CV$-category $\CA$ to the object of $\mathfrak{E}$
determined be the $\CV$-category $\mathrm{JS}\CA = [\mathfrak{S}, \CA]$ of
species in $\CA$, equipped with $L$-weighted derivation the $D^*$ just as in
Example~\ref{specder} with the codomain $\CV$ replaced by~$\CA$.
Since $\mathrm{U}$ is strong monoidal, the biadjunction
$\mathrm{U} \dashv_{\mathrm{bi}} \mathrm{JS}$
is monoidal.
Consequently the biadjunction lifts to one between the 2-categories of monoidales
in $\mathfrak{E}$ and $\CV\text{-}\mathrm{Cat}_{L,+}$.
Indeed $\mathrm{U}$ is pseudocomonadic.
\end{Remark}

\section{The iterated tensor product again}\label{titpa}

Observe the following simple reindexing of \eqref{Lten}.

\begin{Proposition}\label{alt2Ltensor}
An alternative definition of $F\otimes^LG$ is
\begin{gather*}
(F\otimes^LG)X = \sum_{V\subseteq U\subseteq X}{L(U\backslash V) \otimes F(U)\otimes G(X\backslash V)}   .
\end{gather*}
\end{Proposition}

This leads us to another formula for the $n$-fold $L$-weighted tensor product.
Def\/ine the {\em modified $n$-filtration set}\footnote{We are ``modifying'' the f\/iltration $U$ of $X$ by equipping it with the extra subsets $V$.}
on any f\/inite set $X$ by
\begin{gather*}
\mathrm{mFil}_nX= \bigl\{(U,V) \, | \, U =  (0=U_0\subseteq U_1\subseteq \dots \subseteq U_{n-1}\subseteq U_n=X  ), \nonumber \\
\hphantom{\mathrm{mFil}_nX= \bigl\{}{} V =  (V_0, V_1, \dots, V_{n-1}  ) \text{ with } V_i\subseteq U_i \text{ for } 0\le i < n \bigr\}  .
\end{gather*}

\begin{Proposition}
An alternative definition of the $n$-fold tensor product \eqref{nfoldLtensor1} is
\begin{gather}
 \otimes^L_n(F_1,\dots,F_n)X  =  \sum_{(U,V)\in \mathrm{mFil}_nX} {L(U_1\backslash V_1) \otimes \dots \otimes L(U_{n-1}\backslash V_{n-1}})  \nonumber
\\
  \phantom{\otimes^L_n(F_1,\dots,F_n)X  = }{} \otimes  F_1 ( U_1\backslash V_0  ) \otimes \dots \otimes F_n ( U_n\backslash V_{n-1} ). \label{nfoldLtensor4}
\end{gather}
\end{Proposition}

\begin{proof}
The formula follows by repeated application of the formula of Proposition~\ref{alt2Ltensor}
in eva\-luating the left bracketing
\begin{gather*}
\big(\cdots \big(F_1\otimes^LF_2\big)\otimes^L \dots \big)\otimes^LF_n
\end{gather*}
at $X$.
\end{proof}

Let us relate the formulas \eqref{nfoldLtensor1} and \eqref{nfoldLtensor4} in the case $n=3$.
A modif\/ied 3-f\/iltration $(U,V)\in \mathrm{mFil}_3X$ of $X$ amounts to subsets $U_1\subseteq U_2\subseteq X$
and $V_1\subseteq U_1, V_2\subseteq U_2$. With this we can def\/ine
\begin{gather*}
 A_1 = V_1\cap V_2   , \qquad A_2 = V_2\backslash U_1\cap V_2   , \qquad A_3 = X\backslash U_2 ,   \\
  A_{12} = U_1\cap V_2 \backslash A_1   , \qquad A_{13} = V_1\backslash A_1   , \qquad  A_{23} = (U_2\backslash U_1) \backslash A_2   , \qquad A_{123} = (U_1\backslash V_1) \backslash A_{12}
\end{gather*}
and verify that $X= A_1+A_2+A_3+A_{12}+A_{13}+A_{23}+A_{123}$.
Conversely, given the partition~$A$ of~$X$, we can def\/ine
\begin{gather*}
U_1 = X\backslash (A_2+A_3 +A_{23})   , \qquad U_2 = X\backslash A_3   , \qquad V_1=A_1+A_{13}   , \qquad V_2 = A_1+A_2+A_{12}   .
\end{gather*}

\section{Tensor products for charades}\label{charades}

The term ``charade'' is intended in the sense of Kapranov \cite[Def\/inition~3.2]{Kapr1995} and
is related to Hall algebras (for example, see~\cite[Section~2]{UCL}).
One has a small abelian (or triangulated) category $\CA$ and looks
at functors def\/ined on the groupoid $\CA_{\mathrm{g}}$ of invertible morphisms in~$\CA$.
A~promonoidal structure is def\/ined on $\CA_{\mathrm{g}}$ using the short exact sequences
(or triangles) of~$\CA$. The functors are tensored using convolution.
Here we will only discuss the case where~$\CA$ is the category of f\/inite vector spaces
over a f\/ixed f\/inite f\/ield~$\mathbb{F}_q$.
We make a conjecture that a family of weighted monoidal structures exists and give some evidence for it.

Motivated by Proposition~\ref{alt2Ltensor}, we consider generalizing the tensor product of~\cite{GLFq}.
Let $\mathfrak{G}_q$ be the groupoid of f\/inite vector spaces over the f\/ield $\mathbb{F}_q$
of cardinality $q$; the morphisms are linear bijections.
We write $V\le U$ to mean $V$ is an $\mathbb{F}_q$-linear subspace of~$U$,
and we write $U/V$ for the quotient space.

To be specif\/ic, take $\CV = \mathrm{Vect}_{\mathbb{C}}$ to be the category of complex vector spaces with all linear functions.

Let $L\colon \mathfrak{G}_q \to \CV$ be a suitable functor: we will consider conditions on it later.

For functors $F, G \colon \mathfrak{G}_q \to \CV$, def\/ine $F\otimes^LG \colon \mathfrak{G}_q \to \CV$ by
\begin{gather}\label{qLtensor}
\big(F\otimes^LG\big)X = \sum_{V\le U\le X}{L(U/V) \otimes F(U)\otimes G(X/ V)}   .
\end{gather}

This leads us to an $n$-fold tensor product in a manner analogous to~\eqref{nfoldLtensor4}.
Def\/ine the {\em modified $n$-flag set} on any f\/inite $\mathbb{F}_q$-vector space $X$ by
\begin{gather*}
\mathrm{mFlg}_nX= \bigl\{(U,V) \, | \, U =  (0=U_0\le U_1\le \dots \le U_{n-1}\le U_n=X  ), \nonumber \\
\hphantom{\mathrm{mFlg}_nX= \bigl\{}{} V =  (V_0, V_1, \dots, V_{n-1}  ) \text{ with } V_i\le U_i \text{ for } 0\le i < n \bigr\}   .
\end{gather*}

Now we put
\begin{gather*}
 \otimes^L_{n}(F_1,\dots,F_n)X
  =   \sum_{(U,V)\in \mathrm{mFlg}_nX} {L(U_1\backslash V_1) \otimes \dots \otimes L(U_{n-1}\backslash V_{n-1}})
\\
  \phantom{\otimes^L_{n}(F_1,\dots,F_n)X =}{} \otimes  F_1 ( U_1\backslash V_0  ) \otimes \dots \otimes F_n ( U_n\backslash V_{n-1} )  .
\end{gather*}
The formula follows by repeated application of \eqref{qLtensor} in evaluating the left bracketing
\begin{gather*}
 \big(\cdots \big(F_1\otimes^LF_2\big)\otimes^L \dots \big)\otimes^LF_n
\end{gather*}
at $X$.

Let us look at the ternary tensor product
\begin{gather*}
\otimes^L_{3}(F,G,H)X = \big(\big(F\otimes^LG\big)\otimes^LH\big)X
\end{gather*}
It is a direct sum over modif\/ied 3-f\/lags $(U,V)$ on $X$;
that is, subspaces
$U_1\le U_2\le X$, $V_1\le U_1$ and $V_2\le U_2$.
From these we can uniquely def\/ine vector spaces $A_S$ for each
$\varnothing \ne S\subseteq \langle 3 \rangle$
via the following diagrams of short exact sequences:
\begin{gather*}
\xymatrix{
A_1 \ar@{{ >}->}[d]_-{}  \ar@{{ >}->}[r]^-{} & U_1\cap V_2 \ar@{{ >}->}[d]_-{} \ar@{->>}[r] & A_{12} \ar@{{ >}->}[d] \\
V_1  \ar@{->>}[d]_-{} \ar@{{ >}->}[r]^-{} & U_1 \ar@{->>}[d]_-{} \ar@{->>}[r]^-{} & U_1/V_1 \ar@{->>}[d]^-{}   \\
A_{13} \ar@{{ >}->}[r]_-{} & U_1/U_1\cap V_2 \ar@{->>}[r]_-{} & A_{123} \\
}
\\ 
\xymatrix{
U_1\cap V_2 \ar@{{ >}->}[d]_-{}  \ar@{{ >}->}[r]^-{} & V_2 \ar@{{ >}->}[d]_-{} \ar@{->>}[r] & A_{2} \ar@{{ >}->}[d] \\
U_1  \ar@{->>}[d]_-{} \ar@{{ >}->}[r]^-{} & U_2 \ar@{->>}[d]_-{} \ar@{->>}[r]^-{} & U_2/U_1 \ar@{->>}[d]^-{}   \\
U_1/U_1\cap V_2 \ar@{{ >}->}[r]_-{} & U_2/V_2 \ar@{->>}[r]_-{} & A_{23} \\
}
\\ 
\xymatrix{
U_2 \ar@{{ >}->}[r]_-{} & X \ar@{->>}[r]_-{} & A_{3}   ,
}
\end{gather*}
from which we see
\begin{gather}\label{directsumdecomp}
X \cong A_1\oplus A_2 \oplus A_3\oplus A_{12}\oplus A_{23}\oplus A_{13}\oplus A_{123}   .
\end{gather}
Note also the isomorphisms
\begin{gather*}
U_1/V_1 \cong A_{12}\oplus A_{123}   , \qquad U_2/V_2 \cong A_{13}\oplus A_{23} \oplus A_{123},  \\
U_1 \cong A_1\oplus A_{12}\oplus A_{13}\oplus A_{123}   , \qquad U_2/V_1 \cong A_{2}\oplus A_{12}\oplus A_{23}\oplus A_{123},    \\
X/V_2 \cong A_{3}\oplus A_{13}\oplus A_{23} \oplus A_{123}   .
\end{gather*}

On the other hand, we can see that the formula for the right bracketing is
\begin{gather*}
 \big(F\otimes^L\big(G\otimes^LH\big)\big)X
  =   \sum_{M\le N\le X, M\le I\le J\le X}{L(N/M)\otimes L(J/I)\otimes FN\otimes G(J/M)\otimes H(X/I)}.
\end{gather*}
We can see that this indexing set also leads to a direct sum decomposition~\eqref{directsumdecomp}
from the following diagrams of short exact sequences:
\begin{gather*}
\xymatrix{
A_1 \ar@{{ >}->}[d]_-{} \ar@{{ >}->}[r]^-{}
& U_1\cap V_2 \ar@{{ >}->}[d]_-{} \ar@{->>}[r] & A_{12} \ar@{{ >}->}[d] \\
V_1 \ar@{->>}[d]_-{} \ar@{{ >}->}[r]^-{} & U_1 \ar@{->>}[d]_-{} \ar@{->>}[r]^-{}
& U_1/V_1 \ar@{->>}[d]^-{} \\
A_{13} \ar@{{ >}->}[r]_-{} & U_1/U_1\cap V_2 \ar@{->>}[r]_-{} & A_{123} \\
}
\\ 
\xymatrix{
I\cap N \ar@{{ >}->}[d]_-{}  \ar@{{ >}->}[r]^-{} & I \ar@{{ >}->}[d]_-{} \ar@{->>}[r] & A_{2} \ar@{{ >}->}[d] \\
J\cap N  \ar@{->>}[d]_-{} \ar@{{ >}->}[r]^-{} & J \ar@{->>}[d]_-{} \ar@{->>}[r]^-{} & J/J\cap N \ar@{->>}[d]^-{}   \\
A_{123} \ar@{^{(}->}[r]_-{} & J/I \ar@{->>}[r]_-{} & A_{23} \\
}
\\ 
\xymatrix{
J+N \ar@{{ >}->}[r]_-{} & X \ar@{->>}[r]_-{} & A_{3}  , \qquad J\cap N \ar@{{ >}->}[r]_-{} & N \ar@{->>}[r]_-{} & A_{13}  .
}
\end{gather*}
Note also the isomorphisms
\begin{gather*}
N/M \cong A_{12}\oplus A_{13}\oplus A_{123}  , \qquad J/I \cong A_{23} \oplus A_{123}, \qquad
N \cong A_1\oplus A_{12}\oplus A_{13}\oplus A_{123}   , \\
 J/M \cong A_{2}\oplus A_{12}\oplus A_{23}\oplus A_{123},\qquad
X/I \cong A_{3}\oplus A_{13} \oplus A_{23} \oplus A_{123}   .
\end{gather*}

In order to have an associativity isomorphism we at least need a canonical isomorphism
\begin{gather*}
 L(A_{12}\oplus A_{123})\otimes L(A_{13}\oplus A_{23} \oplus A_{123})
  \cong   L(A_{12}\oplus A_{13}\oplus A_{123}) \otimes L(A_{23} \oplus A_{123})   .
\end{gather*}
We do have such an isomorphism if $L\colon \mathfrak{G}_q\to \CV$ takes direct sums to tensor products;
of course, direct sum of $\mathbb{F}_q$-vector spaces is neither product nor coproduct in~$\mathfrak{G}_q$.
This is still merely evidence that the desired associativity isomorphism should exist: it
is not a complete def\/inition.

Recall from \cite{GLFq} that $\mathfrak{G}_q$ has a braided promonoidal structure.
The convolution structure on $[\mathfrak{G}_q,\CV]$ arising from this
(as per Day~\cite{DayPhD, DayConv})
is precisely the tensor product $F\otimes^J G$ where $JX = I$ for $X=0$ and $JX=0$ for
$X\ne 0$.

\begin{Conjecture}
If $L$ is braided strong promonoidal then~\eqref{qLtensor} defines a
monoidal structure $\otimes^L$ on $[\mathfrak{G}_q,\CV]$.
\end{Conjecture}

Should this be the case, the tensor $\otimes^L$ on $[\mathfrak{G}_q,\CV]$ would be obtained
from quite an interesting promonoidal structure on~$\mathfrak{G}_q$.
A short sequence
\begin{gather}\label{spes}
\begin{split}
& \xymatrix{
A \ar@{{ >}->}[r]^-{f} & X \ar@{->>}[r]^-{g} & B
}
\end{split}
\end{gather}
in $\mathrm{Vect}_{\mathbb{F}_q}$ might be called {\em short pre-exact}
when $f$ is a monomorphism, $g$ is an epimorphism and $\mathrm{ker}g \le \mathrm{im}f$.
Write $\mathrm{Spes}(A,B;X)$ for the set of such~$(f,g)$.
Put
\begin{gather*}
\mathrm{P}(A,B;X) = \sum_{(f,g)\in \mathrm{Spes}(A,B;X)}{L (\mathrm{im}(g\circ f) )}   .
\end{gather*}
This $\mathrm{P} \colon \mathfrak{G}_q^{\mathrm{op}}\times \mathfrak{G}_q^{\mathrm{op}}\times \mathfrak{G}_q \lra \CV$,
def\/ined on morphisms in the obvious way, would give the promonoidal structure in question.
The term $L (\mathrm{im}(g\circ f) )$ measures the failure of the sequence~\eqref{spes}
to be exact.

\section{The dimension sequence}\label{tds}

Following on from Section~\ref{charades}, we take $F\in [\mathfrak{G}_q,\mathrm{Vect}_{\mathbb{C}}]$
and def\/ine its {\em dimension sequence} $\dim F\in \mathbb{Z}^{\mathbb{N}}$ by
\begin{gather*}
(\dim F)n = \dim \big(F\big(\mathbb{F}_q^n\big)\big)   .
\end{gather*}

This inspires an algebra structure on $A^{\mathbb{N}}$ for any $k$-algebra $A$.
We assume we have $\lambda \in k$ as before, but also some integer~$q$
(not necessarily a prime power).
As in~\cite{GLFq}, we use
\begin{gather*}
\phi_n(q) = \big(q^n-1\big)\big(q^{n-1}-1\big)\cdots (q-1)   .
\end{gather*}
We def\/ine
\begin{gather*}
{n \brack r,s}_{q} = \frac{\phi_n(q)}{\phi_r(q)\phi_s(q)}, \qquad {n \brack r,s,t}_{q} = \frac{\phi_n(q)}{\phi_r(q)\phi_s(q) \phi_t(q)}, \qquad \dots   .
\end{gather*}

For $f,g \in A^{\mathbb{N}}$, put
\begin{gather*}
f \cdot^{\lambda}_q g = \sum_{r+s+t=n}{{n \brack r,s,t}_{q} \lambda^t f(r+t) g(s+t)}  .
\end{gather*}
The calculations of Section~\ref{charades} show that this is associative at least when $A=\mathbb{Z}$,
$q$~is a prime power and $\lambda = \dim L(\mathbb{F})$.

More generally, I claim $A^{\mathbb{N}}$ is an associative $k$-algebra.

\begin{Proposition}\label{dim}
$\dim (F\otimes^L G) = \dim F \cdot^{\lambda}_q \dim G$
\end{Proposition}

\subsection*{Acknowledgements}

I am grateful to the referees for their careful work and, in particular, for pointing out the references
\cite{AFM2015, AM2006, MoreiraPhD}. The author gratefully acknowledges the support of Australian Research Council Discovery Grant DP130101969.

\pdfbookmark[1]{References}{ref}
\LastPageEnding

\end{document}